\documentclass{amsart}

\usepackage[initials]{amsrefs}

\usepackage{hyperref}

\usepackage{xcolor}

\newcommand{\R}{\mathbb{R}}

\newtheorem{theorem}{Theorem}
\newtheorem{lemma}{Lemma}

\newtheorem{proposition}{Proposition}

\theoremstyle{definition}
\newtheorem{remark}{Remark}

\title{On the multiplicity of arrangements of congruent zones on the sphere}

\author{A. Bezdek}
\address{Department of Mathematics and Statistics, Auburn University, 221 Parker Hall, Auburn, AL 36849, U.S.A.}
\email{bezdean@auburn.edu}

\author{F. Fodor}
\address{Department of Geometry, Bolyai Institute, University of Szeged, Aradi v\'ertan\'uk tere 1, 6720 Szeged, Hungary}
\email{fodorf@math.u-szeged.hu}

\author{V. V\'\i gh}
\address{Department of Geometry, Bolyai Institute, University of Szeged, Aradi v\'ertan\'uk tere 1, 6720 Szeged, Hungary}
\email{vigvik@math.u-szeged.hu}

\author{T. Zarn\'ocz}
\address{Bolyai Institute, University of Szeged, Aradi v\'ertan\'uk tere 1, 6720 Szeged, Hungary}
\email{tzar25@gmail.com}

\begin{document}

\begin{abstract}
Consider an arrangement of $n$ congruent zones on the $d$-dimensional unit sphere $S^{d-1}$, where a zone is the intersection of an origin symmetric Euclidean plank with $S^{d-1}$. We prove that, for sufficiently large $n$, it is possible to arrange $n$ congruent zones of suitable width on $S^{d-1}$ such that no point belongs to more than a constant number of zones, where the constant depends only on the dimension and the width of the zones. 
Furthermore, we also show that it is possible to cover $S^{d-1}$ by $n$ congruent zones such that each point of $S^{d-1}$ belongs to at most $A_d\ln n$ zones, where the $A_d$ is a constant that depends only on $d$. This extends the corresponding $3$-dimensional result of Frankl, Nagy and Nasz\'odi \cite{FNN2016}.
Moreover, we also examine coverings of $S^{d-1}$ with congruent zones under the condition that each point of the sphere belongs to the interior of at most $d-1$ zones. 
\end{abstract}

\maketitle

\section{Introduction and Results} 
A {\em plank} in the Euclidean $d$-space $\R^d$ is a closed region bounded by two parallel hyperplanes. The width of a plank is the distance between its bounding hyperplanes. The famous plank problem of Tarski \cite{T1932} seeks the minimum total width of $n$ planks that can cover a convex body $K$ (a compact convex set with non-empty interior). 

In this paper we consider a spherical variant of the plank problem{\color{red},} which originates from L. Fejes T\'oth \cite{LFT1973}. Following Fejes T\'oth, we call the parallel domain of spherical radius $w/2$ of a great sphere $C$  on the $d$-dimensional unit sphere $S^{d-1}$ a {\em spherical zone}, or zone for short. $C$ is the central great sphere of the zone and $w$ is its (spherical) width. For positive integers $d\geq 3$ and $n$, let $w(d,n)$ denote the smallest number such that the union of $n$ zones of width $w(d,n)$ can cover $S^{d-1}$. 
Fejes T\'oth asked in \cite{LFT1973} the exact value of $w(3,n)$. He conjectured that in the optimal configuration the central great circles of the zones all go through an antipodal pair of points and they are distributed evenly, so in this case $w(d,n)=\pi/n$. Fejes T\'oth's conjecture was verified for $n=3$ (Rosta \cite{Rosta1972}) and $n=4$ (Linhart \cite{Linhart1974}).  Fodor, V\'igh and Zarn\'ocz \cite{FVZ2016} gave a lower bound for $w(3,n)$ that is valid for all $n$. 
Finally, Jiang and Polyanskii \cite{JP2017} completely solved L. Fejes T\'oth's conjecture by proving for all $d$, that in order to cover $S^{d-1}$ by $n$ (not necessarily congruent) zones, the total width of the zones must be at least $\pi$, and that the optimal configuration is essentially the same as conjectured by L. Fejes T\'oth. For the most recent developement in this topic we refer to \cite{GKP2022}.

In this paper, we examine arrangements of congruent zones on $S^{d-1}$ from the point of view  of multiplicity. The multiplicity of an arrangement is the maximal number of zones with nonempty intersection. We seek to minimize the multiplicity for given $d$ and $n$ as a function of the common width of the zones. It is clear that for $n\geq d$, the multiplicity of any arrangement with $n$ congruent zones is at least $d$ and at most $n$. Notice that in the Fejes T\'oth configuration the multiplicity is exactly $n$, that is, maximal.


Our first result is a slight strengthening of the former  fact.

\begin{theorem}\label{lowerbound}
Let $d\geq 2$ and $n\geq 1$ be integers, and let $S^{d-1}$ be covered by the union of $n$ congruent zones. If each point of $S^{d-1}$ belongs to the interior of at most $d-1$ zones, then $n\leq d$. Moreover, if $n=d$, then the $d$ 
zones are pairwise orthogonal.
\end{theorem}

Note that Theorem~\ref{lowerbound} does not imply that the multiplicity of a covering of $S^{d-1}$ with $n\geq d+1$ congruent zones has to be larger than $d$. For example, one can cover $S^2$ with $4$ zones such that the multiplicity is $3$. For this, consider three zones whose central great circles pass through a pair of antipodal points (North and South Poles) and are distributed evenly. Let the central great circle of the fourth zone be the Equator. The common width can be chosen in such a way that there is no point contained in more than three zones. Also, one can arrange five zones such that the multiplicity is still $3$. We start with the previously given four zones, and take another copy of the zone whose central great circle is the Equator. Now slightly tilt these two zones. It is not difficult to see that the multiplicity of the resulting configuration is $3$. The details are left to the reader.

Now, we turn to the question of finding upper bounds on the multiplicity of arrangements of zones on $S^{d-1}$. Let $\alpha:{\mathbb N}\to (0,1]$ be a positive real function with $\lim_{n\to\infty}\alpha(n)=0$. For a positive integer $d\geq 3$, 
let $m_d=\sqrt{2\pi d}+1$. Let $k:{\mathbb N}\to\mathbb{N}$ be a function
that satisfies the limit condition

\begin{equation}\label{eq:limit}
\limsup_{n\to\infty} \alpha(n)^{-(d-1)}\left (\frac{e\;C_d^*\; n\;\alpha(n)}{k(n)}\right )^{k(n)}=\beta<1,
\end{equation}

where 
$$C_d^*=\frac{4(m_d+1)(d-1)\kappa_{d-1}}{d\kappa_{d}}.$$

\begin{theorem}\label{thm:ujmain}
For each positive integer $d\geq 3$, and any real function $\alpha(n)$ described above, for sufficiently large $n$, there exists an arrangement of $n$ zones of spherical half-width $m_d\alpha(n)$ on $S^{d-1}$ such that no point of $S^{d-1}$ belongs to more than $k(n)$ zones.
\end{theorem}

The following statement provides an upper bound on the multiplicity of coverings of the $d$-dimensional unit sphere by $n$ congruent zones.  

\begin{theorem}\label{thm:main}
For each positive integer $d\geq 3$, there exists a positive constant $A_d$ such that for sufficiently large $n$, there is a covering of $S^{d-1}$ by $n$ zones of half-width $m_d\frac{\ln n}{n}$ such that no point of $S^{d-1}$ belongs to more than  $A_d\ln n$ zones. 
\end{theorem}


In the next statement we find some pairs of functions $\alpha(n)$ and $k(n)$ that satisfy condition \eqref{eq:limit}, and thus have the property stated in Theorem~\ref{thm:ujmain}.
\begin{proposition}
	\begin{itemize} 
		\item[i)] If $\alpha(n)=n^{-(1+\delta)}$ for some $\delta>0$, then  there exists a constant $c(\delta)$ such that $k(n)=c(\delta)$ satisfies \eqref{eq:limit}. Moreover, if $\delta>d-1$, then $k(n)=d$ is suitable.
		\item[ii)] The pair of functions $\alpha(n)=\frac{1}{n}$, $k(n)=B_d\frac{\ln n}{\ln\ln n}$ satisfies \eqref{eq:limit} for some suitable constant $B_d$. 
	\end{itemize}	
\end{proposition}

We note that Theorem~\ref{thm:main} and an implicit version of Theorem~\ref{thm:ujmain} were proved by Frankl, Nagy and Nasz\'odi for the case $d=3$, see \cite[Theorem~1.5 and Theorem~1.6]{FNN2016} and also the proof of Theorem~1.5 therein. They provided two independent proofs, one of which is a probabilistic argument and the other one uses the concept of VC-dimension. We further add that the weaker upper bound of $O(\sqrt{n})$ on the minimum multiplicity of coverings of $S^2$ was posed as an exercise in the 2015 Mikl\'os Schweitzer Mathematical Competition \cite{KN2015} by A. Bezdek, F. Fodor, V. V\'\i gh and T. Zarn\'ocz (cf. Exercise 7). 

Our proofs of Theorems~\ref{thm:ujmain} and \ref{thm:main} are based on the probabilistic argument of Frankl, Nagy and Nasz\'odi \cite{FNN2016}, which we modified in such a way that it works in all dimensions. In the course of the proof we also give an upper estimate for the constant $A_d$ whose order of magnitude is $O(d)$.

Obviously, there is a big gap between the lower and upper bounds for the multiplicity of coverings of $S^{d-1}$ by congruent zones. At this time, it is an open problem if the minimum multiplicity of coverings of $S^{d-1}$ by $n$ congruent zones is bounded or not, and it also remains unknown whether the multiplicity is monotonic in $n$, see the corresponding conjectures of Frankl, Nagy and Nasz\'odi on $S^2$  in \cite[Conjectures~4.2 and 4.4]{FNN2016}. 

The multiplicity of coverings of $\R^d$ and $S^{d}$ by convex bodies have already been investigated. In their classical paper, Erd\H os and Rogers \cite{ER1961} proved, using a probabilistic argument, that $\R^d$  ($d\geq 3$) can be covered by translates of a given convex body such that the density of the covering is less than 
$d\log d+d\log\log d+4n$ and no point of $\R^d$ belongs to more than 
$e(d\log d+d\log\log d+4n)$ translates. Later, F\"uredi and Kang \cite{FK2008} gave a different proof of the result of Erd\H os and Rogers using John ellipsoids and the Lov\'asz Local Lemma. B\"or\"oczky and Wintsche \cite{BV2003} showed that
for $d\geq 3$ and $0<\varphi< \pi/2$, $S^d$ can be covered by 
spherical caps of radius $\varphi$ such that the multiplicity of the covering is at most $400d\ln d$. 

\section{Proofs}
\subsection{Proof of Theorem~\ref{lowerbound}}
Assume that $n\geq d$ and that $S^{d-1}$ is covered by $n$ congruent zones such that no point of $S^{d-1}$ belongs to the interior of more than $d-1$ zones. Then the $n$ central great spheres of the zones divide $S^{d-1}$ into convex spherical polytopes. As no $d$  central great spheres can be incident with a point of $S^{d-1}$,  every such polytope is simple, that is, each of its vertices is incident with exactly $d-1$ facets. 

In contrast to the Euclidean space, the insphere (the maximum radius sphere contained in the polytope) of every convex spherical polytope is uniquely determined. In order to see this, assume, on the contrary, that there exists a spherical convex polytope $\mathcal{P}$ such that it contains two spheres, $B$ and $B'$, of maximal radius. Then the (spherical) convex hull of $B$ and $B'$ also belongs to $\mathcal{P}$, and it is clear that there is a sphere contained in $\mathcal{P}$ centred at the midpoint of the geodesic segment connecting the centres of $B$ and $B'$ that has a larger radius than $B$ and $B'$.  

We note that the inradius of each polytope produced by the central great spheres of the zones is less than or equal to the half-width of the zones.



Before we formulate the key lemma of the proof, we state a well-known fact.

\begin{proposition}\label{dplus1point}
    Among $d+1$ (pairwise different) points on $S^{d-1}$ there are $d$ that are contained in an open half-sphere or there are $d$ that are on a great-sphere.
\end{proposition}

\begin{proof}
    Let $x_1, \ldots, x_d$ be  points on $S^{d-1}$ that are neither contained in an open half-sphere nor are on a great-sphere. Then the (Euclidean) simplex $\Delta$, whose vertices are $x_1, \ldots, x_d$, contains the origin in its interior. Consider those half-spaces in $\R^d$ that are bounded by the affine hulls of the facets and do not contain the origin. These half-spaces clearly cover $S^{d-1}$. Thus at least one such half-space contains $x_{d+1}$, and we are done.
\end{proof}

Now we formulate the main lemma used in the proof.

\begin{lemma}\label{lemma-incircle}
Every simple spherical polytope $\mathcal{P}$ on $S^{d-1}$ with more than $d$ facets and inradius $r$ contains a point $P$ whose distance from at least $d$ facets is less than $r$. 
\end{lemma}
 
\begin{proof} 
Assume that $\mathcal{P}$ is a simple spherical polytope with $m\geq d+1$ facets.
Let $V$ be a vertex of  $\mathcal{P}$ and let the $d-1$ facets incident with $V$ be $F_1,\ldots, F_{d-1}$. Note that the great spheres containing $F_1,\ldots, F_{d-1}$ also contain the antipodal point $V'$ to $V$. Since $V$ is a vertex of $\mathcal{P}$, it is clear that  $V'\notin\mathcal{P}$.  

Let $B_V$ be a ball of maximal volume that is contained in $\mathcal{P}$ and which is tangent to $F_1,\ldots, F_{d-1}$. If there is no other facet of $\mathcal{P}$ tangent $B_V$, then $B_V$ is centred at the midpoint of the geodesic segment in $\mathcal{P}$ connecting $V$ and $V'$ that is equidistant to the great spheres containing $F_1,\ldots, F_{d-1}$. In this case, there exists a point $P$ on this segment on the same side of the midpoint as $V'$ that is closer to $F_1,\ldots, F_{d-1}$ and to one more facet $F$ of $\mathcal{P}$ than the the radius of $B_V$. As the radius of $B_V$ is at most $r$, we have found the desired point $P$.

Now, we may assume that for each vertex $V$ of $\mathcal{P}$ the maximum radius ball $B_V$ that is tangent to all facets incident $V$ is also tangent to at least one more facet $F_V$, that is different from $F_1,\ldots, F_{d-1}$. If there exists a vertex $V$ of $\mathcal{P}$ such that $B_V$ is not the insphere of $\mathcal{P}$, then its centre $P$ is the desired point. Thus, we may assume that the spheres $B_V$ are all equal to the (unique) insphere $B$ of $\mathcal{P}$ for each vertex $V$. This means that the $B$ is tangent to all facets of $\mathcal{P}$, or equivalently, $\mathcal{P}$ is circumscribed around its insphere. From the assumptions it follows that there are at least $d+1$ points of tangency on $B$, hence by Proposition~\ref{dplus1point}
there are $d$ points of tangency that are contained in an open half-sphere $B^+$ of $B$ or are on a great-subsphere of $B$. In the former case, there exists a point $P$ on the geodesic segment with one endpoint at the centre of $B$ and with direction to the centre of $B^+$ that is closer than $r$ to the facets whose points of tangency are in $B^+$, and the lemma readily follows. In the latter case, all facets of $\mathcal{P}$ are incident with $V$ and thus $\mathcal{P}$ is not simple.

\end{proof}

Now we return to the proof of Theorem~\ref{lowerbound}. We start with proving the first statement of Theorem~\ref{lowerbound} by induction on the dimension $d$.

Clearly, we may assume that $n\geq d$. Note that under the assumptions of Theorem~\ref{lowerbound}, any $d$ of the $n$ central great spheres of the zones divide $S^{d-1}$ into spherical simplices as the intersection of their hyperplanes is equal to the origin. On the other hand, Lemma~\ref{lemma-incircle} guarantees that the spherical polytopes determined by the central great spheres of the $n$ zones do not have more than $d$ facets, therefore they are all spherical simplices.

Let $d=2$. 
Then the vertices and sides of the triangular domains determined by the central great circles of the zones form a planar graph $G$ on $S^2$. The number $v$ of vertices is $2{n\choose 2}$, and the number of edges is $2n(n-1)$. By Euler's formula, the number $f$ of faces (the number of spherical triangles) is 
$$f=e+2-v=n^2-n+2.$$
Furthermore, the degree of each vertex is four, thus $4v=3f$, which yields that
$$n^2-n-6=0.$$
The only positive root of the above quadratic equation is $n=3$. 

Now, assume that the first statement of Theorem~\ref{lowerbound} holds in all dimensions $k$ for $3\leq k\leq  d-1$. According to Lemma~\ref{lemma-incircle}, the central great spheres $S_1,\ldots, S_n$ of the zones divide $S^{d-1}$ into spherical simplices such that no point of $S^{d-1}$ is incident with more than $d-1$ of $S_1,\ldots, S_n$. Then $S_1\cap S_2,\ldots, S_1\cap S_n$ are all great spheres (of dimension $n-2$) on $S_1$ that divide $S_1$ into spherical simplices (of dimension $d-2$) such that no point of $S_1$ is incident with more than $d-2$ of $S_1\cap S_2,\ldots, S_1\cap S_n$. Therefore, by the induction hypothesis, $n-1=d-1$, and thus $n=d$. This concludes the proof of the first statement of Theorem~\ref{lowerbound}.

Next, we prove the second statement of Theorem~\ref{lowerbound}.

{The case $d=2$ is obvious, so we may assume that $d\geq 3$ arbitrary. Assume that the central great spheres of the zones $Z_1, \ldots, Z_{d-1}$ meet in the antipodal pair of points $N$ and $S$ that we still call the North Pole and the South Pole, respectively, and let $E$ denote the Equator (the great sphere of $S^{d-1}$ whose hyperplane is orthogonal to the line through $N$ and $S$, denoted by $NS$). Note that $U=Z_1\cap \ldots \cap Z_{d-1}$ is an (unbounded) polyhedron (in $\R^d$) that is symmetric with respect to the line $NS$.} 

{Each zone $Z_i$ is bounded by a pair of parallel hyperplanes. If we choose one hyperplane from every such pair (for $i=1,\ldots, d-1$), the intersection of the $d-1$ chosen hyperplanes is a line $t$ parallel to $NS$. These $d-1$ hyperplanes cut $\mathbb \R^d$ into polyhedral regions ($d-1$-hyperoctants). If we reflect $U$ through $t$, then the reflection $U'$ is in a region that is obviously not covered by any $Z_i$.}

{Now, if $t$ intersects the interior of $B^d$, then this uncovered region cuts out a spherical region $R$ from $S^{d-1}$ that is not covered by any $Z_i$. Clearly, $t\cap S^{d-1}$ consists of two extreme points of $R$, that are closest to $N$ and $S$, respectively, or equivalently, that are furthest from $E$.}

{Consider all the line segments $t\cap B^d$. These come in pairs that are symmetric with respect to $NS$, and some of them might be empty. Obviously the $d$th zone $Z_d$ should cover all such segments. However, if $Z_d$ contains a segment in its interior, then $Z_d$ contains an interior point of $U\cap S^{d-1}$, that contradicts our initial assumption. As the segments are all parallel, it follows that they share the same length, and hence $Z_d$ is centered on $E$. This yields that $Z_d$ is orthogonal to the zones $Z_1, \ldots, Z_{d-1}$ and, as $Z_d$ was arbitrary, the proof of the theorem is complete.}

\subsection{Proof of Theorem~\ref{thm:ujmain}}
For two points $P,Q\in S^{d-1}$, their spherical distance is the length of the 
shorter unit-radius circular arc on $S^{d-1}$ that connects them. We denote the spherical distance by $d_S(P,Q)$.
  
Let $0<\omega\leq \pi/2$. We say that the points $P_1,\ldots, P_m\in S^{d-1}$ form a {\em saturated set} for $\omega$ if the spherical distances $d_S(P_i,P_j)\geq\omega$ for all $i\neq j$ and no more points can be added such that this property holds. 
Investigating the dependence of $m$ on $d$ and $\omega$ is a classical topic in the theory of packing and covering; for a detailed overview of known results in this direction see, for example, the survey paper by Fejes T\'oth and Kuperberg \cite{FTGK}.

It is clear that $m$ is of the same order of magnitude as $\omega^{-(d-1)}$. In the next lemma, we prove a somewhat more precise statement. Although the content of the lemma is well-known, we give a proof because we need 
inequalities for $m$ with exact constants in subsequent arguments, and also 
for the sake of completeness.   
Let $\kappa_d$ denote the volume of the $d$-dimensional unit ball $B^d$. 

\begin{lemma}
Let $0<\varepsilon <1$. Then there exists $0<\omega_0\leq \pi/2$ depending on $\varepsilon$ with the following property. Let $0<\omega<\omega_0$, and let $P_1,\ldots, P_m$ be a saturated point set for $\omega$. 
Then
$$(1+\varepsilon)^{-1} \frac{d\kappa_d}{\kappa_{d-1}} \omega^{-(d-1)}\leq m\leq (1+\varepsilon) \frac{8^{\frac{d-1}{2}}d\kappa_d}{\kappa_{d-1}} \omega^{-(d-1)}.$$ 
\end{lemma}

\begin{proof}
The following formula is known for the 
surface area $S(t)$ of a cap of height $t$ of $S^{d-1}$, cf. \cite[formula (3.4) on p. 796]{BFH13},
$$\lim_{t\to 0+} S(t)\; t^{-\frac{d-1}{2}}=2^{\frac{d-1}{2}}\kappa_{d-1}.$$
Therefore, there exists $0<t_0=t_0(\varepsilon)$ such that for 
all $0<t<t_0$ it holds that

$$(1+\varepsilon)^{-1} 2^{\frac{d-1}{2}}\kappa_{d-1}\leq S(t)\; t^{-\frac{d-1}{2}}\leq (1+\varepsilon) 2^{\frac{d-1}{2}}\kappa_{d-1}.$$
Furthermore, let $0<\omega_0=\omega_0(\varepsilon)$ be such that $t_0=1-\cos\omega_0$.

The spherical caps of (spherical) radius $\omega/2$ centred at $P_1,\ldots, P_m$ form a packing on $S^{d-1}$, and the spherical caps of radius $\omega$ form a covering of $S^{d-1}$. In view of the above inequalities for the surface area of caps, we obtain that for $0<\omega<\omega_0$ it holds that
$$m(1+\varepsilon)^{-1}\; 2^{\frac{d-1}{2}}\kappa_{d-1}\left (1-\cos \frac\omega 2\right )^{\frac{d-1}{2}} \leq d\kappa_d\leq m (1+\varepsilon) 2^{\frac{d-1}{2}}\kappa_{d-1}(1-\cos \omega)^{\frac{d-1}{2}}.$$

By simple rearrangement we get that
$$(1+\varepsilon)^{-1} \frac{d\kappa_d}{2^{\frac{d-1}{2}}\kappa_{d-1}(1-\cos\omega)^{\frac{d-1}{2}}}\leq m\leq (1+\varepsilon) \frac{d\kappa_d}{2^{\frac{d-1}{2}}\kappa_{d-1}\left (1-\cos\frac\omega 2\right )^{\frac{d-1}{2}}}.$$
Now, we use that for $0<x<1$, it holds that $x^2/4<1-\cos x<x^2/2$, which
follow simply from the Taylor series of $\cos x$, and obtain the desired inequalities
$$(1+\varepsilon)^{-1}\frac{d\kappa_d}{\kappa_{d-1}}\; \omega^{-(d-1)}\leq m\leq (1+\varepsilon) \frac{8^{\frac{d-1}{2}}d\kappa_d}{\kappa_{d-1}}\; \omega^{-(d-1)}.$$ 
\end{proof}
We denote a spherical zone of (spherical) half-width $t$ by $\Pi(t)$. 
Since, for $d\geq 2$ and small $t$, it holds that
$$
S(\Pi(t)) = 2(d - 1)\kappa_{d-1}\int_0^t \cos^{d-2} (\tau)d\tau, 
$$
it follows that 
$$\lim_{t\to 0^+} S(\Pi(t))\cdot t^{-1}=2(d-1)\kappa_{d-1}.$$
Let $\varepsilon>0$. Then there exists $t_1=t_1(\varepsilon)>0$ such that
for $0<t<t_1$ the following holds
$$(1+\varepsilon)^{-1} 2(d-1)\kappa_{d-1}\; t\leq S(\Pi(t))\leq (1+\varepsilon) 2(d-1)\kappa_{d-1}\; t.$$

Let $\alpha(n)$ be a given positive function with $\lim_{n\to\infty}\alpha(n)=0$. 
From now on,  we fix $\varepsilon=1$, set $m_d=\sqrt{2\pi d}+1$, and assume $n$ to be sufficiently large.

Let $Q_1,\ldots, Q_m$ be a saturated set of points on $S^{d-1}$ such that 
$d_S(Q_i,Q_j)\geq \alpha(n)/2$ for any $i\neq j$. It follows from Lemma~1 that
\begin{align*}
m &\leq 2 \frac{8^{\frac{d-1}{2}}d\kappa_d}{\kappa_{d-1}} (\alpha(n)/2)^{-(d-1)}\\
&=2 \frac{2^{\frac{d-1}{2}}d\kappa_d}{\kappa_{d-1}} \alpha(n)^{-(d-1)}\\
&= c_d\; \alpha(n)^{-(d-1)}.
\end{align*}

Consider 
$n$ independent random points from $S^{d-1}$ chosen according to the uniform
probability distribution and consider the corresponding spherical zones\\ $\Pi_1,\ldots, \Pi_n$ of (spherical) half-width $m_d\alpha(n)$ whose poles are these points. Furthermore, let $\Pi_i^-$, $\Pi_i^+$ be the corresponding planks of half-width $(m_d-1)\alpha(n)$ and $(m_d+1)\alpha(n)$, respectively.

Now, we are going to estimate the probability of the event that there exists a point $p$ on $S^{d-1}$ which belongs to at least $k=k(n)$ zones. 
The probability that a point $p\in S^{d-1}$ belongs to a spherical plank $\Pi_i^+$ can be estimated from above as follows.
$${\mathbb P}(p\in\Pi_i^+)\leq \frac{4(m_d+1)(d-1)\kappa_{d-1}}{d\kappa_d}\;\alpha(n)=C_d^*\; \alpha(n).$$
Note that $C_d^*=O(d)$ as $d\to\infty$. 

Then
\begin{align*}
&{\mathbb P}(\exists p\in \Pi_{i_1}\cap\dots\cap\Pi_{i_k}:\text{ for some }1\leq i_1<\ldots <i_k\leq n)\\
\leq\ &{\mathbb P}(\exists Q_j\in\Pi^+_{i_1}\cap\dots\cap\Pi^+_{i_k}:\text{ for some }1\leq i_1<\ldots <i_k\leq n)\\
\leq\ &m\cdot{\mathbb P}(Q_1\in\Pi^+_{i_1}\cap\dots\cap\Pi^+_{i_k}:\text{ for some }1\leq i_1<\ldots <i_k\leq n)\\
\leq\ &m\cdot\binom n{k(n)}\left(C_d^*\;\alpha(n)\right)^{k(n)}\\
\leq\ &c_d\;\alpha(n)^{-(d-1)}\binom n{k(n)}\left(C_d^*\;\alpha(n)\right)^{k(n)}
\end{align*}

An application of the Stirling-formula (cf. Page 10 of \cite{FNN2016}) yields that
\begin{equation}\label{binom}
\binom nk\leq C \frac{n^n}{k^k(n-k)^{n-k}}
\end{equation}
for some suitable constant $C>0$. 

Then applying \eqref{binom} we get that
\begin{align}
&c_d\; \alpha(n)^{-(d-1)}\binom n{k(n)}\left(C_d^*\;\alpha(n)\right)^{k(n)}\notag\\
\leq &c_d\;\alpha(n)^{-(d-1)}  \cdot C  \frac{n^n (n-k(n))^{k(n)}}{\left(k(n)\right)^{k(n)-n}}
\left(C_d^*\;\alpha(n)\right)^{k(n)}\notag\\
\leq &\tilde{c_d}\;\alpha(n)^{k(n)-d+1}\left (\frac{n}{k(n)}\right )^{k(n)} (e\cdot C_d^*)^{k(n)}\label{eq:final}\notag\\
&=\tilde{c_d}\;\alpha(n)^{-(d-1)} \left (\frac{e\;C_d^*\; n\;\alpha(n)}{k(n)}\right )^{k(n)}.
\end{align}

\label{limit}

By \eqref{eq:limit} we obtain
$$\limsup_{n\to\infty}{\mathbb P}(\exists p\in \Pi_{i_1}\cap\dots\cap\Pi_{i_k}:\text{ for some }1\leq i_1<\ldots <i_k\leq n)<1,$$
therefore the probability of the event that no point of $S^{d-1}$ belongs to at least $k(n)$ zones is positive for sufficiently large $n$. This finishes the proof of Theorem~\ref{thm:ujmain}.
 
\subsection{Proof of Theorem~\ref{thm:main}}
Let $\alpha(n)=\frac{\ln n}{n}$, and let $k(n)=A_d\ln n$, where $A_d$ be a suitable positive constant  that satisfies the following equation 
$$\left (\frac{C_d^*}{x}\right )^{x}=e^{-d-x}.$$
Then
\begin{align}
	\eqref{eq:limit}&=\lim_{n\to\infty}\tilde{c_d}\frac{n^{d-1}}{(\ln n)^{d-1}}\cdot n^{A_d}\left (\frac{C_d^*}{A_d}\right )^{A_d\ln n}=0.\label{eq:final}
\end{align}

Furthermore, in this case the probability that an arbitrary fixed point $p$ of $S^{d-1}$ is in $\Pi_i^-$ (for a fixed $i$) is
$${\mathbb P}(p\in\Pi_i^-)\geq 2^{-1}\cdot \frac{2(d-1)\kappa_{d-1}}{d\kappa_d}\cdot (m_d-1)\alpha(n).$$
Using the inequality 
$\frac{\kappa_{d-1}}{d\kappa_d}>\frac{1}{\sqrt{2\pi d}}$ (cf. Lemma~1 in  \cite{BGW1982}), we obtain that
$${\mathbb P}(p\in\Pi_i^-)\geq \frac{(m_d-1)(d-1)}{\sqrt{2\pi d}}\cdot \frac{\ln n}{n}=(d-1)\frac{\ln n}{n}$$
Thus, the probability that $\cup_{1}^{n}\Pi_i$ does not cover $S^{d-1}$ satisfies 
\begin{align*}
{\mathbb P}(S^{d-1}\not\subseteq\cup_{1}^{n}\Pi_i) &\leq 
{\mathbb P}(\exists Q_j\notin\cup_{1}^{n}\Pi_i^-)\\
&\leq m\cdot {\mathbb P}(Q_1\notin\cup_{1}^{n}\Pi_i^-)\\
&\leq  c_d \left (\frac{n}{\ln n}\right )^{d-1}\cdot
\left (1-(d-1)\frac{\ln n}{n}\right )^n\\
&\leq 2c_d \left (\frac{1}{\ln n}\right )^{d-1}
\end{align*}
for a sufficiently large $n$. Therefore
\begin{equation}\label{eq:firstcase}
\lim_{n\to\infty}{\mathbb P}(S^{d-1}\not\subseteq\cup_{1}^{n}\Pi_i)=0.
\end{equation}

Thus, taking into account \eqref{eq:final} and \eqref{eq:firstcase}, the probability of the event that all $S^{d-1}$ is covered by the zones and no point of $S^{d-1}$ belongs to more than $A_d\ln n$ zones is positive for sufficiently large $n$. This finishes the proof of Theorem~\ref{thm:main}.

\begin{remark}
	We note that $A_d=O(d)$ as $d\to\infty$. Clearly, $A_d$ can be lowered slightly by taking into account all the factors of \eqref{eq:final}.	
\end{remark}

\begin{remark}
	We further note that one can obtain the result of Theorem~\ref{thm:main}  
	with the help of
	Theorem~1.6 
	of \cite{FNN2016} using the VC-dimension of hypergraphs; for more details we refer to the discussion in \cite{FNN2016} after Theorem~1.6. However, as this alternate proof is less geometric in nature, we decided to describe the more direct probabilistic proof of Theorem~\ref{thm:main}. We leave the proof of Theorem~\ref{thm:main} that uses the VC-dimension to the interested reader. Furthermore, the direct probabilistic argument provides an explicit estimate of the involved constant $A_d$, as well. 
\end{remark}

\subsection{Proof of Proposition~1}
Let $\alpha(n)=\frac{1}{n^{1+\delta}}$ for some $\delta>0$. If $k=k(n)>(d-1)/\delta+d-1$, then 
\begin{align*}
\limsup_{n\to\infty} \alpha(n)^{-(d-1)}\left (\frac{e\;C_d^*\; n\;\alpha(n)}{k(n)}\right )^{k(n)}&=\lim_{n\to\infty} n^{(1+\delta)(d-1)}
\left (\frac{e\;C_d^*\;n^{-\delta}}{k}\right )^k\\
&=\lim_{n\to\infty} n^{(1+\delta)(d-1)-\delta k}=0.
\end{align*}
This means that in this case, for sufficiently large $n$, it can be guaranteed that one can arrange $n$ zones of half-width $m_d\alpha(n)$ on $S^{d-1}$ such that no point belongs to more than $k=const.$ zones, and the value of $k$ only depends on $d$ and $\delta$. Moreover, if $\delta>d-1$, then $k=d$ suffices. Of course, in this case the zones cannot cover $S^{d-1}$. This proves i) of Proposition~1.

Now, let $\alpha(n)=\frac{1}{n}$, and let $k(n)=B_d\frac{\ln n}{\ln\ln n}$, where 
$B_d>\max \{e\;C_d^*, d-1\}$ is a positive constant. Then
\begin{align*}
\limsup_{n\to\infty} \alpha(n)^{-(d-1)}\left (\frac{e\;C_d^*\; n\;\alpha(n)}{k(n)}\right )^{k(n)}&=\lim_{n\to\infty}n^{d-1}\left(\frac{e\;C_d^*\;\ln\ln n}{B_d\;\ln n}\right )^{B_d\frac{\ln n}{\ln\ln n}}\\
&\leq \lim_{n\to\infty} \left(\frac{n^{\frac{(d-1)\ln\ln n}{B_d\ln n}}\ln\ln n}{\ln n}\right )^{B_d\frac{\ln n}{\ln\ln n}}=0,
\end{align*}
as
\begin{align*}
&\lim_{n\to\infty}\frac{n^{\frac{(d-1)\ln\ln n}{B_d\ln n}}\ln\ln n}{\ln n}\\
=&\lim_{n\to\infty}\exp\left (\frac{d-1}{B_d}\ln\ln n+\ln\ln\ln n-\ln\ln n\right )=0.
\end{align*}

This finishes the proof of part ii) of Proposition~1. The above statement is interesting because $\alpha(n)=\frac 1n$ is the smallest order of magnitude for the half-width of the zones for which one can possibly have a covering. 
\begin{remark}
We note that the $d=3$ special case of part ii) of Proposition~1 was explicitly proved by
Frankl, Nagy and Nasz\'odi in \cite{FNN2016} (cf. Theorem~4.1) in a slightly different form both by the probabilistic method and using VC-dimension. We also note that the general $d$-dimensional statement of part ii) of Proposition~1 may also be proved from Theorem~1.6 of \cite{FNN2016}.
\end{remark}

\section{Acknowledgements}
{The authors thank Bal\'azs Csik\'os (E\"otv\"os University, Budapest, Hungary) for the suggestion to use induction on the dimension in the proof of Theorem~\ref{lowerbound}.}

Research of A. Bezdek was partially supported by ERC Advanced Research Grant 267165 (DISCONV).

F. Fodor, V. V\'\i gh, and T. Zarn\'ocz were supported by the {National Research, Development and Innovation Office – NKFIH 116451 grant.} 

{F. Fodor was also supported by the National Research, Development and Innovation Office – NKFIH K134814 grant.}

V. V\'\i gh was supported by the J\'anos Bolyai Research Scholarship of the Hungarian Academy of Sciences and by the Hungarian NKFIH grant FK135392.

{This research was supported by project TKP2021-NVA-09. Project no. TKP2021-NVA-09 has been implemented with the support provided by the Ministry of Innovation and Technology of Hungary from the National Research, Development and Innovation Fund, financed under the TKP2021-NVA funding scheme.}

\end{document}